\documentclass[12pt,reqno]{amsart}
\usepackage{a4}
\usepackage{mathrsfs}
\usepackage{color}
\usepackage{enumitem}
\usepackage{graphicx}
\usepackage{subfigure}
\usepackage[colorlinks=true,
linkcolor=webgreen,
filecolor=webgreen,
citecolor=webgreen]{hyperref}
\definecolor{webgreen}{rgb}{0,0,1}
\definecolor{recrown}{rgb}{1,.2,.6}

\begin{document}
\newtheorem{theorem}{Theorem}
\newtheorem{corollary}[theorem]{Corollary}
\newtheorem{lemma}[theorem]{Lemma}
\newtheorem{proposition}[theorem]{Proposition}
\theoremstyle{definition}
\newtheorem{example}{Example}
\newtheorem{examples}{Examples}
\newtheorem*{notation}{Notation}
\theoremstyle{remark}
\newtheorem*{remarks}{\bf Remarks}
\theoremstyle{thmx}
\newtheorem{thmx}{\bf Theorem}
\renewcommand{\thethmx}{\text{\Alph{thmx}}}
\newtheorem{lemmax}{Lemma}
\renewcommand{\thelemmax}{\text{\Alph{lemmax}}}
\theoremstyle{thmx}
\newtheorem{property}{\bf Property}
\renewcommand{\theproperty}{\text{\Alph{property}}}
\theoremstyle{definition}
\newtheorem*{definition}{Definition}
\numberwithin{theorem}{section}
\numberwithin{example}{section}
\newtheorem*{remark}{\bf Remark}
\newcommand{\C}{\mathcal{C}_{x,y}}
\newcommand{\s}{\mathbb{S}}
\newcommand{\rr}{\mathbb{R}_0}

\title[]{\bf Round and sleek topological spaces}
\markright{}
\subjclass[2010]{Primary 54E25; 54E35; 46A55; 52A07; 46B20}
\keywords{Round map; Sleek map; Round topology, Sleek topology; Topological space; Continuous map; Topological vector space}
\author{Jitender Singh$^{\dagger}$}
\address{$~^\dagger$ Department of Mathematics, Guru Nanak Dev University, Amritsar-143005, India}
\footnotetext[2]{$^{}$Corresponding author: jitender.math@gndu.ac.in}
\date{}
\maketitle
\begin{abstract}
In this paper, we introduce round and sleek topological spaces and study their properties.
\end{abstract}
\section{Introduction}
 Let $X$ be a metrizable space whose topology is induced by a metric $d$ on $X$. The metric $d$ is called a round metric for $X$ if closure of each open ball is the corresponding closed ball in $X$, and a metrically round space is a metrizable space that admits a round metric for its topology \cite{Na}. Similarly, the metric $d$ is called a sleek metric for $X$ if interior of each closed ball is the corresponding open ball in $X$. A metrically  sleek space is a metrizable space that admits a sleek metric for its topology \cite{JSTD2022}. All normed linear spaces and all strictly convex linear metric space are metrically round as well as sleek, but a metric space having at least two points and having an isolated point fails to be metrically round or sleek.
 In general, the class of metrically round metric spaces and the class of metrically sleek metric spaces turn out to be different in the sense that there exist metrizable spaces which are metrically round as well as sleek, round but never sleek, sleek but never round, and neither round nor sleek \cite{JSTD2021,JSTD2022}. Recently in \cite{JSTDarXiv}, some results were proved for metrically round and sleek subsets of linear metric spaces and metric spaces in the subspace topology. For example, an algebraically convex subset of a normed linear space or a linear metric space is metrically round in the subspace topology. Also, an externally convex round metric space is metrically sleek (See \cite[Theorem 3.10]{JSTDarXiv}).
The present work is motivated from the following characterization theorem.
\begin{thmx}[\cite{JSTD2022}]\label{thA} Let $(X,d)$ be a metric space having at least two points.
\begin{enumerate}[label=(\alph*)]
\item\label{Aa} The metric $d$ is not round for $X$ if and only if there exists an open set $U$ in $X$ and $x\in X\setminus U$ for which the map $d(x,\cdot):U\rightarrow \mathbb{R}$ has a minimum.
\item\label{Ab} The metric $d$ is not sleek for $X$ if and only if there exists an open set $U$ in $X$ and $x \in U$ for which the map $d(x,\cdot):U\rightarrow \mathbb{R}$ has a maximum.
\end{enumerate}
\end{thmx}
Theorem \ref{thA} tells us that if $d$ is a round or sleek metric for $X$ (having at least two points), then for every $x\in X$, the map $d(x,\cdot)$ is not locally constant.

In the present paper, we extend round and sleek notions from metric spaces to arbitrary topological spaces appropriately and prove analogous results for such spaces. As we shall see in this paper, many of the results on round and sleek properties of metrizable spaces extend to arbitrary topological spaces. The paper is organized as follows. In Section \ref{sec:2}, we define round topology and sleek topology on an arbitrary set and study their properties. In Section \ref{sec:3}, we discuss round and sleek subsets of topological vector spaces in the subspace topology. Some examples are discussed in Section \ref{sec:4}.
\section{Round and sleek topological spaces}\label{sec:2}
Unless otherwise specified, the symbol $X$ will denote a topological space having at least two points and $X^2$, the cartesian product $X\times X$ equipped with the product topology. Also, $\rr$ will denote the set of all nonnegative real numbers endowed with the topology $\rr$ inherits as a subspace of real line.

Theorem \ref{thA} allows us the following extended notions.
\begin{definition}[Round map and round space]
Let  $f:X^2\rightarrow \rr$ be a continuous map. We call $f$, a round map for $X$ if for every proper nonempty open set $U$ in $X$ and $x\in X\setminus U$, the map $f(x,\cdot):U\rightarrow \rr$ has no minimum. A topological space admitting  a round map will be called a round space. The topology of a round space will be called a \emph{round topology}.
\end{definition}
\begin{definition}[Sleek map and sleek space]
Let $f:X^2\rightarrow \rr$ be a continuous map. We call $f$, a sleek map for $X$ if for every nonempty open set $U$ in $X$ and $x \in U$, the map $f(x,\cdot):U\rightarrow \rr$ has no maximum. A topological space admitting  a sleek map will be called a sleek space. The topology of a sleek space will be called a \emph{sleek topology}.
\end{definition}
Note from the preceding definitions, that if a topologically space $X$ is metrically round (resp. sleek), then $X$ is a round (resp. sleek) space. Now we have the following results.
\begin{proposition}\label{p1} Let $X$ be a topological space having at least two points.
\begin{enumerate}[label=(\alph*)]
\item\label{p1a} If $X$ has an isolated point, then no $f$ is round or sleek for $X$.
\item\label{p1b} If $X$ is a disjoint union of two nonempty subspaces $A$ and $B$, where $A$ is closed and $B$ is compact, then no $f$ is round for $X$.
\item\label{p1c} If $X$ is compact, then no $f$ is sleek for $X$.
\item\label{p1d} If $X$ has at least one proper nonempty open subset, then there exists a continuous non-round $f$ for $X$.
\item\label{p1e} There exists a continuous non-sleek $f$ for $X$.
\item\label{p1f} The indiscrete topology on $X$ is a round topology but not a sleek topology.
\end{enumerate}
\end{proposition}
\begin{proof}
Let $f: X^2\rightarrow \rr$ be any continuous map.

$(a)$ If $y\in X$ is an isolated point of $X$, then $\{y\}$  is open in $X$, and so, for any $x$ in $X$ with $x\neq y$, the map $f(x,\cdot):\{y\}\rightarrow \rr$ has the minimum value $f(x,y)$. So, $f$ is not a round map for $X$. Similarly, the map $f(y,\cdot):\{y\}\rightarrow \rr$ has the maximum value $f(y,y)$, which shows that $f$ is not a sleek map for $X$.

$(b)$ By the hypothesis, $B=X\setminus A$ and $A$ is closed in $X$. So, the set  $B$ is open in $X$. Let $x\in A$. Since $B$ is compact, by extreme value theorem, the continuous map $f(x,\cdot):B\rightarrow \rr$ has a minimum. So, $f$ is not a round map for $X$.

$(c)$ Since $X$ is compact and for any $x\in X$, the map $f(x,\cdot):X\rightarrow \rr$ is continuous, by extreme value theorem, $f(x,\cdot)$ has a maximum, and so, $f$ is not a sleek map for $X$.

$(d)\& (e)$ To prove $(d)$ with the given hypothesis and $(e)$,  we observe that any constant $f$ is neither round nor sleek for $X$.

$(f)$ Since there is no proper nonempty open subset of $X$ in the indiscrete topology, any continuous map $f:X^2\rightarrow \rr$ is a round map for $X$. So, the indiscrete topology on $X$ is a round topology. Since $X$ is compact in the indiscrete topology, by \ref{p1c}, no $f$ can be a sleek map for $X$.
\end{proof}
\begin{proposition}
 Let $X$ be a set having at least two points. Let $\tau$ be a topology of finite cardinality on $X$. The topology $\tau$ is not a sleek topology. If $\tau$ is not the indiscrete topology, then $\tau$ is never a round topology.
\end{proposition}
\begin{proof} Since $\tau$ is of finite cardinality, the topological space $(X,\tau)$ is compact. By Proposition \ref{p1}\ref{p1c}, the topology $\tau$ is never sleek. Now assume that $\tau$ is not the indiscrete topology. If $f:X^2\rightarrow \rr$ is a continuous map, then for any $x\in X$, the map $f(x,\cdot):X\rightarrow \rr$ is locally constant, and so, $f$ is never a round map for $(X,\tau)$. This proves that $\tau$ is not a round topology on $X$.
\end{proof}
\begin{proposition}\label{p5}
    If $X$ is  a round (resp. sleek) topological space, then there exists a bounded round  (resp. sleek) map for $X$.
\end{proposition}
\begin{proof}
    Let $f:X^2\rightarrow \rr$ be a continuous map. Let
    \begin{eqnarray*}
    \bar{f}(x,y)=\frac{f(x,y)}{1+f(x,y)},~\text{for all}~x,y\in X,
    \end{eqnarray*}
Clearly, $\bar{f}$ is a bounded continuous function on $X^2$. We observe that for all $x,y,u\in X$,
\begin{eqnarray*}
f(x,y)\leq f(x,u)~ \text{if and only if}~ \frac{f(x,y)}{1+f(x,y)}\leq \frac{f(x,u)}{1+f(x,u)},
\end{eqnarray*}
that is, $\bar{f}(x,y)\leq \bar{f}(x,u)$. Consequently, $f$ is not a round (resp. sleek) map for $X$ if and only if so is $\bar{f}$ for $X$, and the proposition follows.
\end{proof}
\begin{proposition}\label{p6}
Let $X_1$ and $X_2$ be topological spaces having no isolated point. For each $i=1,2$, let $f_i:X_i^2\rightarrow \rr$ be continuous. Let $F:X_1\rightarrow X_2$ be a continuous bijection  for which  $f_1(x,y)<f_1(x,z)$ implies  $f_2(F(x), F(y))<f_2(F(x),F(z))$ for $x,y,z\in X_1$. If $f_1$ is a round (resp. sleek) map for $X_1$, then $f_2$  is a round (resp. sleek) map for $X_2$.
\end{proposition}
\begin{proof}
We prove the proposition only in the round case, since the proof is similar for the other case. Assume that $f_2$ is not a round map for $X_2$. Then there exists an open set $V\subset X_2$, $F(x)\in X_2\setminus V$, $F(y)\in V$, such that $f_2(F(x),F(y))\leq f_2(F(x),F(u))$ for all $u\in U=F^{-1}(V)$, where by the continuity of $F$, the set $U$  is open in $X_1$ such that $y\in U$ and $x\in X_1\setminus U$. By the hypothesis, $f_1(x,y)\leq f_1(x,u)$ for all $u\in U$. Consequently, $f_1$ is not a round map for $X$.
\end{proof}
\begin{corollary}\label{c2}
  For each $i=1,2$, let $X_i$ be a topological space having no isolated point, and $f_i:X_i^2\rightarrow \rr$ be continuous. If there exists a homeomorphism $F:X_1\rightarrow X_2$ and a positive real number $\kappa$ such that
  \begin{eqnarray*}
  f_2(F(x),F(y))=\kappa f_1(x,y),~\text{for all}~x,y\in X_1,
  \end{eqnarray*}
then $f_1$ is a round (resp. sleek) map for $X_1$ if and only if $f_2$ is a round (resp. sleek) map for $X_2$.
\end{corollary}
\begin{proof} 
For $x,y,z\in X_1$, we have
      \begin{eqnarray*}
  f_2(F(x),F(y))-f_2(F(x),F(z))=\kappa (f_1(x,y)-f_1(x,z)),
  \end{eqnarray*}
  which shows that $f_1(x,y)<f_1(x,z)$ if and only if $f_2(F(x),F(y))<f_2(F(x),F(z))$. Now applying Proposition \ref{p6} for the continuous bijections $F$ and $F^{-1}$, we have  the desired result.
\end{proof}
\begin{proposition}\label{p2} Let $X$ be a topological space, and let $f:X^2\rightarrow \rr$ be a continuous map.
\begin{enumerate}[label=(\alph*)]
  \item\label{p2a} If $f$ is a round   map for $X$, then so is $f$ for each proper nonempty open subset of $X$. The converse holds if $X$ is Hausd\"orff.
  \item\label{p2b} If $f$ is a sleek   map for $X$, then so is $f$ for each proper nonempty open subset of $X$. The converse holds if $X$ is a $T_1$-space.
\end{enumerate}
\end{proposition}
\begin{proof} In view of Proposition \ref{p1}, we can assume without loss of generality that $X$ has no isolated point.

$(a)$ If there exists a nonempty proper open subset $P$ of $X$ such that $f$ is not a round map for $P$,  then there exists an open subset $Q$ of $X$ with $p\in P\setminus Q$ and $q\in P\cap Q$ for which $f(p,q)\leq f(p,r)$ for all $r\in P\cap Q$. Consequently, the map $f(p,\cdot):P\cap Q\rightarrow \rr$ has the minimum value $f(p,q)$ on the open subset $P\cap Q$ of $X$.  So, $f$ is not a round map for $X$.

To prove the converse, we assume that $X$ is Hausd\"orff. Let there be an open subset $U$ of $X$, $x\in X\setminus U$, and $y\in U$ for which the following holds:
\begin{eqnarray}\label{e1}
f(x,y)\leq f(x,u)~\text{for all}~u\in U.
\end{eqnarray}
Since $X$ is Hausd\"orff, we can choose an open subset $V$ of $X$ containing the point $x$ such that $V$ is disjoint from $U$. Since $X$ has no isolated point, there exists $z\in V$ with $z\neq x$. In view of the fact that $X$ is  Hausd\"orff, the set $\{x\}$ is closed in $X$. Consequently, $W:=U\cup V\setminus \{z\}$ is a proper nonempty open subset of  $X$. Observe that $U$ is a nonempty proper open subset of $W$ for which $x\in W\setminus U$ and the continuous map $f(x,\cdot): U\rightarrow \rr$ satisfies \eqref{e1}, and so, $f$ is not a round map for $W$.

$(b)$ If there exists a nonempty proper open subset $P$ of $X$ such that $f$ is not a sleek map for $P$,  then there exists an open subset $Q$ of $X$ and $p,q\in P\cap Q$ for which $f(p,r)\leq f(p,q)$ for all $r\in P\cap Q$. Consequently, the map $f(p,\cdot):P\cap Q\rightarrow \rr$ has the maximum value $f(p,q)$ on the open subset $P\cap Q$ of $X$, and so, $f$ is not a sleek map for $X$.

Conversely, let $X$ be a $T_1$-space. Let there be an open subset $U$ of $X$, $x,y\in U$ for which the following holds:
\begin{eqnarray}\label{e2}
f(x,u)\leq f(x,y)~\text{for all}~u\in U.
\end{eqnarray}
Since $X$ is a $T_1$-space having no isolated point, the set $U\setminus\{x,y\}$ is a nonempty proper open subset of $X$. Let $z\in U\setminus\{x,y\}$. Then $U\setminus\{z\}$ is a nonempty proper open subset of $X$
 for which the map $f(x,\cdot): U\setminus \{z\}\rightarrow \rr$ satisfies \eqref{e2}, and so, $f$ is not a sleek map for $U\setminus \{z\}$.
\end{proof}
\begin{proposition}
    A dense subspace of a round (resp. sleek) topological space is round (resp. sleek).
\end{proposition}
\begin{proof}We will prove the Proposition only for the case of roundness, since the proof is similar for the other case.
So, let $Y$ be a dense subspace of a round topological space $X$. Let $f$ be a round map for $X$. If possible, suppose that $f$ is not a round map for $Y$. Then there exists an open set $O$ in $X$, $x\in Y\setminus O$, and $y\in Y\cap O$ for which $f(x,y)\leq f(x,u)$ for all $u\in Y\cap O$. Let $z\in O$. Since $Y$ is dense in $X$, we must have $\overline{Y\cap O}=\bar{O}$, and so, $z\in \overline{Y\cap O}$. Consequently, there exists a sequence $(z_n)$ of points of $Y\cap O$ converging to $z$. By the continuity of $f$,  we have $f(x,z_n)\rightarrow f(x,z)$ as $n\rightarrow \infty$, where $f(x,y)\leq f(x,z_n)$, since $z_n\in Y\cap O$ for all $n$. We then have $f(x,y)\leq f(x,z)$. Thus, the map $f(x,\cdot): O\rightarrow \rr$ has the minimum value $f(x,y)$ where $y\in O$ and $x\in X\setminus O$. This shows that $f$ is not a round map for $X$, a contradiction.
\end{proof}
\begin{definition}
Let $X$ be a topological space and $f:X^2\rightarrow \rr$ be a continuous map satisfying $f(x,x)=0$ for all $x\in X$. For  $x\in X$ and $r>0$, we define the open ball $B_f(x,r)$ and the closed ball $B_f[x,r]$ centered at $x$ and radius $r\geq 0$ in $X$ as follows:
\begin{eqnarray*}
B_f(x,r)=\{y\in X~\mid~f(x,y)<r\};~B_f[x,r]=\{y\in X~\mid~f(x,y)\leq r\}.
\end{eqnarray*}
We observe that $B_f(x,0)=\emptyset$. Also,  $B_f[x,0]=\{y\in X~|~f(x,y)=0\}$, which is closed by the continuity of $f(x,\cdot)$. The balls $B_f(x,0)$ and $B_f[x,0]$  will be called degenerate balls and the remaining other balls in $X$ will be called non-degenerate balls. For $r>0$, using the continuity of $f(x,\cdot)$ again, we have
\begin{eqnarray*}
B_f(x,r)=f(x,\cdot)^{-1}([0,r)),
\end{eqnarray*}
which is an open set containing $x$ in $X$. Similarly, the closed ball $B_f[x,r]$ is closed in $X$. We also denote by $\bar{B}_f(x,r)$, the closure of the open ball $B_f(x,r)$ and $B_f^{\circ}[x,r]$ as the interior of the closed ball $B_f[x,r]$ in $X$.

We observe that unlike in metrizable spaces, the collection of all nonempty open balls in $X$ need not be a basis for a topology on $X$.
\end{definition}
Our next result, that is, Theorem \ref{p3} below establishes a connection between the round property (resp. the sleek property) with the equality between ``closure of each non-degenerate open ball and the corresponding open ball'' (resp. ``interior of each non-degenerate closed ball and the corresponding open ball'') in a given topological space.
\begin{proposition}\label{p3}
  Let $X$ be a topological space having at least two points. Let there be a continuous map $f:X^2\rightarrow \rr$ such that $f(x,x)=0$ for each $x\in X$. 
  \begin{enumerate} [label=(\alph*)]
    \item\label{p3a} If $f$ is a round map for $X$, then $\bar{B}_f(x,r)=B_f[x,r]$ for all $x\in X$ and  $r>0$. The converse holds if $f(x,y)=0$ for   $x,y\in X$ implies $x=y$.
    \item\label{p3b} The map $f$ is a sleek map for $X$ if and only if  $B_f^\circ[x,r]=B_f(x,r)$ for all $x\in X$ and  $r>0$.
  \end{enumerate}
\end{proposition}
\begin{proof} We prove the contrapositive of each of the given statements. Assume without loss of generality that $X$ has no isolated point.

$(a)$ If there exists $x\in X$ and $r>0$ for which $\bar{B}_f(x,r)\neq B_f[x,r]$, then there exists $y\in B_f[x,r]$ for which $d(x,y)=r$ but $y\not\in \bar{B}_f(x,r)$. So, $U=X\setminus \bar{B}_f(x,r)$ is an open set containing $y$ such that $x\in (X\setminus U)=\bar{B}_f(x,r)$ for which $f(x,y)\leq f(x,u)$ for all $u\in U$, and so, $f$ is not a round map for $X$. Conversely, assume that $f$ is not a round map for $X$. Then there exists an open set $U$ in $X$, $x\in X\setminus U$, and $y\in U$ for which $f(x,y)\leq f(x,u)$ for all $u\in U$. Since $x\neq y$, by the hypothesis  $f(x,y)>0$ so that $B_f(x,f(x,y))\neq\emptyset$. We observe that $B_f(x,f(x,y))\cap U=\emptyset$, which shows that  $y\not\in \bar{B}_f(x,f(x,y))$. Consequently, we have $\bar{B}_f(x,f(x,y))\neq B_f[x,f(x,y)]$.

$(b)$ Now suppose there exists $x\in X$ and $r>0$ for which $B_f^\circ[x,r]\neq B_f(x,r)$. Then there exists $y\in B_f^\circ [x,r]$ for which $d(x,y)=r$. Consequently, $B_f^\circ [x,r]$ is an open set containing $x$ and $y$ such that $f(x,u)\leq f(x,y)$ for all $u\in B_f^\circ [x,r]$. So, $f$ is not a sleek map for $X$. Conversely, assume that $f$ is not a sleek map for $X$. Then there exists an open set $U$ in $X$, $x,y\in U$ for which $f(x,u)\leq f(x,y)$ for all $u\in U$, and so,  $U\subseteq B_f[x,f(x,y)]$. Observe that $f(x,y)>0$, otherwise if $f(x,y)=0$, then $B_f[x,f(x,y)]=\{x\}$, and so, $\emptyset \neq U\subseteq \{x\}$, that is, $U=\{x\}$, which contradicts the fact that $X$ has no isolated point. So, we must have $f(x,y)>0$. Consequently, by the hypothesis we have  $x\neq y$. Since $y\in U$ and $U$ is an open set such that $U\subseteq B_f[x,f(x,y)]$, we have $y\in B_f^{\circ}[x,f(x,y)]$ whereas $y\not\in B_f(x,f(x,y))$. So, we have $B_f^{\circ}[x,f(x,y)]\neq B_f(x,f(x,y))$.
\end{proof}
\begin{proposition}\label{p8}
If $f$ is a round map for a $T_1$-space $X$, then for $x,y\in X$, $f(x,y)=0$ implies $x=y$.
\end{proposition}
\begin{proof}
Assume the contrary that $f(x,y)=0$ for $x\neq y$. Since, $X$ is a $T_1$-space, there exists an open set $V$ containing $y$ such that $V$ does not contain $x$.  Then the map $f(x,\cdot): V\rightarrow \rr$ has the minimum value $f(x,y)=0$, which shows that $f$ is not a round map for $X$. This contradicts the hypothesis.
\end{proof}
The converse of Proposition \ref{p8} is not true. To see this, we consider the subspace  $X=[0,1] \cup [2,3]$ of real line. Clearly, $X$ is a $T_1$-space. Now if we take $f(x,y)=|x-y|$ for all $x,y\in X$, then $f(x,y)=0$ implies $x=y$ but by Proposition \ref{p1}\ref{p1b}, we see that $f$ is not a round map for $X$.
\begin{corollary}\label{c3}
Let $X$ be a $T_1$-space, and let $f:X^2\rightarrow \rr$ be a continuous map satisfying $f(x,x)=0$ for all $x\in X$. Then $f$ is a round map for $X$ if and only if $\bar{B}_f(x,r)=B_f[x,r]$ for all $x\in X$ and $r>0$.
\end{corollary}
\begin{proof}
 Follows from Propositions \ref{p3}\ref{p3a} and \ref{p8}.
\end{proof}
\begin{proposition}\label{p9}
Let $X$ be a topological space and $f:X^2\rightarrow \rr$ be a continuous map.
For an index set $J$, let $(Y_\alpha)_{\alpha\in J}$ be a family of subspaces of $X$, such that $f$ is  a round (resp. sleek) map for the subspace $Y_\alpha\cup Y_\beta$ for all $\alpha,\beta\in J$. Then $f$ is a round (resp. sleek) map for $\cup_{\alpha\in J}Y_\alpha$.
\end{proposition}
\begin{proof} Let $Y=\cup_{\alpha\in J}Y_\alpha$. Assume that $f$ is not a round map for $Y$.
Then there exists an open set $O$ in $X$, $x\in Y\setminus O$, and $y\in Y\cap O$ for which $f(x,y)\leq f(x,u)$ for all $u\in U$ so that $x\not\in O$ and $y\in \cup_{\in J}(Y_\alpha\cap O)$. Let $\alpha,\beta\in J$ be such that $x\in Y_\alpha$ and $y\in Y\beta$. Then $y\in (Y_\alpha \cup Y_\beta)\cap O$ and $x\not\in (Y_\alpha \cup Y_\beta)\cap O$ such that $f(x,y)\leq f(x,u)$ for all $u\in (Y_\alpha \cup Y_\beta)\cap O$ since $(Y_\alpha \cup Y_\beta)\cap O\subseteq Y\cap O$. So, $f$ is not a round map for the subspace $Y$. The proof for the case of sleekness is similar, and so, we omit the proof.
\end{proof}
\begin{proposition}\label{p7}
A locally connected Hausd\"orff space $X$ is round (resp. sleek) if and only if
for every pair of distinct connected components $C_1$ and $C_2$ of $X$, the subset $C_1\cup C_2$ of $X$ is round (resp. sleek) in the subspace topology.
\end{proposition}
\begin{proof}
If $X$ is locally connected, then each component of $X$ is open and so is their union. So, if $X$ is round (resp. sleek), then by Proposition \ref{p2}, union of any two components of $X$ is   round (resp. sleek) in the subspace topology.

We will prove the converse part only in the round case, since the proof is similar for the sleek case. So, let  $f:X^2\rightarrow \rr$ be a continuous non-round map for $X$. Then there exists an open subset $O$ of $X$, $x\in X\setminus O$, and $y\in O$ such that $f(x,y)\leq f(x,u)$ for all $u\in O$. As $X$ is locally connected, there exist connected neighborhoods  $U$ and $V$ of $x$ and $y$, respectively, such that $U\subseteq O$ and $V\subseteq O$. Consequently, $U\cup V\subseteq O$. Since $X$ is Hausd\"orff, we can assume without loss of generality that $U\cap V=\emptyset$. Now if $C_1$ and $C_2$ are the connected components of $X$ such that $x\in C_1$ and $y\in C_2$, then $U\subseteq C_1$ and $V\subseteq C_2$. If we define $C=C_1\cup C_2$, then $C$ is open in $X$, and so, $V$ is an open subset of $C$ such that $x\in U=C\setminus V$, $y\in V$ and $f(x,y)\leq f(x,u)$ for all $u\in V$. So, by Proposition \ref{p2}\ref{p2a}, $f$ is not a round map for $C$.
\end{proof}
\begin{proposition}\label{p10}
Let ${(X_\alpha)}_{\alpha\in J}$ be a family of topological spaces indexed by a nonempty  set $J$.  Let $X=\prod_{\alpha\in J}X_\alpha$ be endowed with the product topology.
\begin{enumerate}[label=(\alph*)]
    \item\label{p10a} The space $X$ is round  if and only if the space $X_\alpha$ is round for every $\alpha\in J$.
    \item\label{p10b} If there exists an index $\beta\in J$ for which $X_\beta$ is a sleek space having at least two points,  then the product space $\prod_{\alpha\in J}X_\alpha$ is sleek.
\end{enumerate}
\end{proposition}
\begin{proof}
For each  $\alpha\in J$, let $\pi_\alpha: X^2\rightarrow X_\alpha^2$, be defined by $\pi_\alpha(x,y)=(x_\alpha,y_\alpha)$ for all $x=(x_\alpha)_{\alpha\in J}, y=(y_\alpha)_{\alpha\in J}\in X$. Clearly, $\pi_\alpha$ is continuous for each $\alpha\in J$.

$(a)$ First assume that $f_\alpha$ is a round map for $X_\alpha$  for every $\alpha\in J$.
Then clearly, $f_\alpha\circ \pi_\alpha$ is a round map for $X$ for every $\alpha\in J$.  Conversely, suppose that $f$ is a round map for $X$. For each $\alpha\in J$, define $g_\alpha(x_\alpha,y_\alpha)=f((x_\beta)_{J},(y_\beta)_{\beta\in J})$, for all $x_\alpha,y_\alpha\in X_\alpha$. Then  $g_\alpha$ is a round map for $X_\alpha$ for each $\alpha$.

$(b)$ Observe that if $f_\beta$ is a sleek map for $X_\beta$ for some $\beta\in J$, then $f_\beta\circ \pi_\beta$ is a sleek map for $X$.
\end{proof}
\section{Round and sleek topological vector spaces}\label{sec:3}
Let $X$ be a real vector space, and let $p$ be a nonzero seminorm on $X$. Let $f(x,y)=p(x-y)$ for all $x,y\in X$. Then the collection of all open balls $(B_f(x,r))_{x\in X,~r>0}$ in $X$ is a basis for a topology on $X$ with respect to which $X$ becomes a topological vector space and $p$ is continuous on $X$. Consequently, $f$ is continuous on $X^2$. We call such a topology on $X$ as the one induced by the seminorm $p$ on $X$.
\begin{proposition}\label{th1}
Let $X$ be a real topological vector space, and let there be a nonzero continuous seminorm $p$ on $X$ inducing the topology of $X$. Let $f(x,y)=p(x-y)$ for all $x,y\in X$. 
\begin{enumerate}[label=(\alph*)]
    \item If $C$ is a convex subset of $X$ in the subspace topology, then  closure of every nonempty open ball in $C$ is the corresponding closed ball in $C$.
    \item Let $C$ be a subset of $X$ in the subspace topology, such that for every pair of distinct points $x$ and $y$ in $C$, there exists a sequence $(z_n)$ of points of $C$ such that
        \begin{eqnarray*}
z_n=(1+t_n)y-t_nx,~0<t_n<1,
        \end{eqnarray*}
and $z_n\rightarrow y$ as $n\rightarrow \infty$. Then interior of every non-degenerate closed ball in $C$ is the corresponding open ball in $C$.
\end{enumerate}
\end{proposition}
\begin{proof} For $x\in C$ and $r>0$, the open ball centered at $x$ and radius $r$ in $C$ is the set $B_{f,C}(x,r)=C\cap B_f(x,r)$. Similarly, the corresponding closed ball in $C$ is the set $B_{f,C}[x,r]=C\cap B_f[x,r]$. Also denote the closure and interior of these balls in $C$ by the symbols
${Cl}_{C}B_{f,C}(x,r)$, and ${Int_C}B_{f,C}[x,r]$, respectively.

$(a)$ Clearly, $Cl{B}_{f,C}(x,r)\subseteq B_{f,C}[x,r]$. Let $y\in B_{f,C}[x,r]$ with $f(x,y)=r$. Let
\begin{eqnarray*}
y_n=\Bigl(1-\frac{1}{n}\Bigr)y+\frac{1}{n}x,~n=2,3,\ldots.
\end{eqnarray*}
Since $C$ is convex, we have $z_n\in C$ for all $n$. We also have
\begin{eqnarray*}
f(x,z_n)=(1-1/n)p(x-y)=(1-1/n)r<r,
\end{eqnarray*}
which shows that $y_n\in B_{f,C}(x,r)$ for all $n>1$. As $f(y,y_n)=r/n\rightarrow 0$ as $n\rightarrow \infty$, the sequence $(y_n)$ of points of $B_{f,C}(x,r)$ converges to $y$. Consequently, $y\in Cl_C{B}_{f,C}(x,r)$, and so, $B_{f,C}[x,r]\subseteq Cl_C{B}_f(x,r)$. Thus,  $Cl_C {B}_f(x,r)=B_{f,C}[x,r]$.

$(b)$ Clearly, $B_{f,C}(x,r)\subseteq Int_C B_{f,C}[x,r]$. If possible, suppose there exists $y\in Int_C B_f[x,r]$ with $f(x,y)=r$. Let
\begin{eqnarray*}
z_n=\Bigl(1+\frac{1}{n}\Bigr)y-\frac{1}{n}x,~n=2,3,\ldots.
\end{eqnarray*}
By the hypothesis, $z_n\in C$ and $f(x,z_n)=(1+1/n)p(x-y)=(1+1/n)r>r$ for all $n$, which shows that $z_n\in C\setminus B_{f,C}[x,r]$. Also, $f(y,z_n)=r/n\rightarrow 0$ as $n\rightarrow \infty$, which shows that the sequence $(z_n)$ of points of the subset $C\setminus B_{f,C}[x,r]$ of $C$ converges to $y$, and so, $y$ is a limit point of the set $C\setminus B_{f,C}[x,r]$. This is absurd since  $Int_C B_{f,C}[x,r]$ is an open subset of $C$ containing $y$, which does not intersect the set $C\setminus B_{f,C}[x,r]$. So, we must have $Int_CB_{f,C}[x,r]=B_{f,C}(x,r)$.
\end{proof}
\begin{corollary}
Let $(p_\alpha)_{\alpha\in J}$ be an indexed family of seminorms on a real vector space $X$ inducing the vector space topology $\tau$ on $X$. For each $\alpha\in J$, if
\begin{eqnarray*}
f_\alpha(x,y)=p_\alpha(x-y) ~\text{for all}~ x,y\in X,
\end{eqnarray*}
then $\bar{B}_{f_\alpha}(x,r)=B_{f_\alpha}[x,r]$ and $B_{f_\alpha}^\circ[x,r]=B_{f_\alpha}(x,r)$ for all $x\in X$ and $r>0$. If the family $(p_\alpha)_{\alpha\in J}$ of seminorms is separating, then $p_\alpha$ is a round map for $X$ for every $\alpha
\in J$.
\end{corollary}
\begin{proof}
For each $\alpha\in J$, the seminorm $p_\alpha$ is continuous on $X$ in the topology $\tau$. In view of Proposition \ref{th1}, take $C=X$ and $p=p_{\alpha}$. If the family $(p_\alpha)_{\alpha\in J}$ of seminorms is separating, then $X$ is Hausd\"orff, and so, by Corollary \ref{c3}, $p_\alpha$ is a round map for $X$ for every $\alpha\in J$.
\end{proof}
\section{Examples}\label{sec:4}
\begin{example}
We show that the usual topology of $\rr$ is round as well as sleek.

The metric  $d$ on $\rr$ defined by $d(x,y)=|x-y|$ for all $x,y\in \rr$ is a round map for $\rr$, and so, $\rr$ is a round space. However, the metric $d$ is not a sleek map for $\rr$, since $[0,2)$ is an open subset of  $\rr$ containing 0 and 1 for which $d(0,\cdot):[0,2)\rightarrow \rr$ has the maximum value $d(0,1)=1$.

Let $f:\rr^2\rightarrow \rr$ be define by $f(x,y)=y$ for all $x,y\in \rr$. Clearly, $f$ is continuous. If $U$ is a nonempty open subset of $\rr$ and $x\in U$, then we claim that $f(x,\cdot):U\rightarrow \rr$ has no maximum. So, assume on the contrary that there exists such an $U$ and $x\in U$ for which $u=f(x,u)\leq f(x,y)=y$ for all $u\in U$ for some $y\in U$. Then $U\subseteq [0,y]$. Since $\{0\}$ is not open in $\rr$, we must have $y>0$. Choose $\epsilon>0$ for which $(y-\epsilon,y+\epsilon)\subseteq U\subseteq [0,y]$. Then each  $z\in (y-\epsilon,y)\cap U$ satisfies $f(x,z)=z<y=f(x,y)$, which contradicts the hypothesis. Thus, $f$ is a sleek map for $\rr$. Here, we note that $f$ is not a round map for $\rr$.
\end{example}
\begin{example}
  Consider the product space $\mathbb{R}_D\times \mathbb{R}_u$, where $\mathbb{R}_D$ and $\mathbb{R}_u$ denote the set of all real numbers endowed with  the discrete topology and the usual topology, respectively. Since every point of $\mathbb{R}_D$ is an isolated point, by Proposition \ref{p1}\ref{p1a}, the space $\mathbb{R}_D$ is neither round nor sleek. On the other hand, the usual metric of real line is a round as well as sleek map for $\mathbb{R}_u$, and so, $\mathbb{R}_u$ is a sleek space. By Proposition \ref{p10}\ref{p10b}, the product space $\mathbb{R}_D\times \mathbb{R}_u$ is sleek.
\end{example}
\begin{example}
Let $\mathbb{R}_K$ denote the set of all real numbers in the $K$-topology $\tau_K$. We show that each of the topologies $\tau_u$ and $\tau_K$ is round as well as sleek.

If we take $d(a,b)=|a-b|$ for all $a,b\in \mathbb{R}$, then clearly, $d$ is a round as well as a sleek map for $\mathbb{R}_u$.

Now let $f(a,b)=|b|$ for all $a,b\in \mathbb{R}$, then clearly $f$ is continuous on each of the topological spaces $\mathbb{R}_u$, $\mathbb{R}_K$, and $\mathbb{R}_\ell$. We show that $f$ is a round map for $\mathbb{R}_K$. If possible, suppose there exists $U\in \tau_K$, $x\in \mathbb{R}_K\setminus U$, and $y\in U$ for which $|y|=f(x,y)\leq f(x,u)=|u|$ for all $u\in U$. Consequently, $(-|y|,|y|)\cap U=\emptyset$, which shows that $0\not\in U$, and so, $U\in \tau_u$. Since $y\in U$, we have $y\neq 0$; and since $U\in \tau_u$, there exists an $\epsilon>0$ with $0<\epsilon<|y|$ for which $(y-\epsilon,y+\epsilon)\subseteq U$. Let
     \begin{eqnarray*}
z_\epsilon=\begin{cases} y-(\epsilon/2), &~\text{if}~y>0;\\
y+(\epsilon/2), &~\text{if}~y<0.
\end{cases}
\end{eqnarray*}
Then $z_\epsilon\in (y-\epsilon,y+\epsilon)\subseteq U$, such that $f(x,z_\epsilon)=|z_\epsilon|=|y|-(\epsilon/2)<|y|$, which is a contradiction. This proves that $f$ is a round map for $\mathbb{R}_K$. We leave it to the reader to verify that  $f$ is a sleek map for $\mathbb{R}_K$.
\end{example}
\begin{example}
Let $f:\mathbb{R}^2\rightarrow \rr$, where
\begin{eqnarray*}
f(x,y)=\begin{cases} |y|,~&\text{if}~y\in \mathbb{Q};\\
    0,~&\text{if}~y\not\in \mathbb{Q}.
\end{cases}
\end{eqnarray*}
Let $\tau$ be the collection $\{f^{-1}(V)~\mid~V ~\text{is open in }~ \rr\}$. Then $\tau$ is the smallest topology on $\mathbb{R}$ in which $f$ is continuous. Observe that $\tau$ is not a $T_1$-topology on $\mathbb{R}$, since no singleton set is closed. Also, every nonempty open subset of  $(\mathbb{R},\tau)$ contains infinitely many rational numbers. Using these observations, we show that $f$ is a sleek map for $(\mathbb{R},\tau)$. Suppose on the contrary that for any continuous map $f:\mathbb{R}^2\rightarrow\rr$, there exists $U\in \tau$ and $x,y\in U$ for which $f(x,u)\leq f(x,y)$ for all $u\in U$. Then $U\subseteq B_f[x,f(x,y)]$. Since $U\cap \mathbb{Q}$ is an infinite set, we must have $f(x,y)>0$, and so,  $y$ is a nonzero rational number so that $f(x,y)=|y|$. Since $y\in U$, and $U$ is open, there exists an $\epsilon>0$ for which
$0<|y|-\epsilon$ and $\pi_2(f^{-1}(|y|-\epsilon,|y|+\epsilon))\subseteq U$, where $\pi_2:\mathbb{R}^2\rightarrow \mathbb{R}$ is the projection map defined by $\pi_2(a,b)=b$ for all $a,b\in \mathbb{R}$, which is an open map, and we note that
\begin{eqnarray*}
\pi_2(f^{-1}((|y|-\epsilon,|y|+\epsilon)))=\{(|y|-\epsilon,|y|+\epsilon)\cup (-|y|-\epsilon,-|y|+\epsilon)\}\cap\mathbb{Q}.
\end{eqnarray*}
If we choose a rational number $q\in (|y|,|y|+\epsilon)$, then $q\in U$ but $f(x,q)=q>|y|=f(x,y)$, which is a contradiction. This proves that $f$ is a sleek map for $(\mathbb{R},\tau)$.

Observe that $f$ is not a round map for $(\mathbb{R},\tau)$, since $U=(\mathbb{R}\setminus \mathbb{Q})\cup (-1,1)$   is an open subset of $(\mathbb{R},\tau)$ and $2\not\in U$ such that $f(2,\cdot):U\rightarrow \rr$ has the minimum value $f(2,\pi)=0$.
\end{example}
\begin{example}
Consider $\mathbb{R}$ in the finite complement topology $\tau$. Let $f:(\mathbb{R},\tau)\times (\mathbb{R},\tau)\rightarrow \rr$ be continuous. Let  $x,y,z\in \mathbb{R}$ such that $y\neq z$. If possible, let $f(x,y)<f(x,z)$, then for $f(x,y)<\alpha<f(x,z)$, by the continuity of $f(x,\cdot)$, each of the sets $U=f(x,\cdot)^{-1}(\alpha,\infty)$ and $V=f(x,\cdot)^{-1}[0,\alpha)$ is a nonempty proper open subset of $\mathbb{R}$ in the finite complement topology. Consequently,  $U\cap V\in \tau$, which is also a proper nonempty subset of $\mathbb{R}$. On the other hand, $U\cap V=f(x,\cdot)^{-1}(\alpha,\infty)\cap f(x,\cdot)^{-1}[0,\alpha)=\emptyset$, which is a contradiction. Similarly, $f(x,y)>f(x,z)$ will lead to a contradiction. We conclude that $f(x,\cdot)$ is a constant map on $(\mathbb{R},\tau)$  for every real number $x$. It follows that the topology $\tau$ is neither round nor sleek.

More generally, there exists a regular space having at least two points on which every real valued continuous map is constant (See \cite{novak}). Such a topological space is neither round nor sleek.
\end{example}

\begin{remark}
The present work lays foundation of the concepts of round and sleek topological spaces. Under a mild condition on $f$, that is, $f(x,x)=0$ for all $x\in X$, one has the equivalence of the notions of  ``$f$ is a round map for $X$'' and ``closure of every non-degenerate open ball (with respect to $f$) is the corresponding closed ball'' provided that $X$ satisfies the $T_1$-separation axiom. On the other hand the condition $f(x,x)=0$ for all $x\in X$ is enough to have the equivalence of  the statements ``$f$ is a sleek map for $X$'' and ``interior of every non-degenrate closed ball (with respect to $f$) is the corresponding open ball''. These concepts behave well with the products. The following natural questions deserve further investigation:
\begin{enumerate}
    \item Which round topological spaces are sleek$?$
    \item Which sleek topological spaces are  round$?$
    \item Is a metrizable round (resp. sleek) space is metrically round (resp. sleek)$?$ 
\end{enumerate}
\end{remark}



\subsection*{Compliance with Ethical Standards}
The author declares that he has no conflict of interest.


\begin{thebibliography}{1}
\bibitem {JSTD2021} J. Singh and T.~D. Narang, Remarks on balls in metric spaces, \emph{The Journal of Analysis} {29} (2021), 1093--1103. \url{https://doi.org/10.1007/s41478-020-00297-z}
\bibitem {JSTD2022} J. Singh and T.~D. Narang, Metrically round and sleek metric spaces, \emph{The Journal of Analysis} 
 (2022), pp 1--17. \url{https://doi.org/10.1007/s41478-022-00459-1}
\bibitem {Na} M.~B. Nathanson, Round metric spaces, \emph{The American Mathematical Monthly}  {82}, no. 7 (1975), 738--741. \url{https://doi.org/10.2307/2318733}
\bibitem {JSTDarXiv} J. Singh and T.~D. Narang, Round and sleek subspaces of linear metric spaces and metric spaces, 
arXiv Preprint (2022), pp1--14. \url{http://arxiv.org/abs/2209.03101v1}
\bibitem{novak} J.~Novak, Regular space, on which every continuous function is constant, \emph{Casopsis Pest. Mat. Fys.}  73 (1948) 58--68.
\end{thebibliography}
\end{document}